\documentclass[]{amsart}
\usepackage[utf8]{inputenc}
\usepackage{amsmath}
\usepackage{amsfonts}
\usepackage{amsthm}
\usepackage{amssymb}
\usepackage{color}
\usepackage{todonotes}
\usepackage{comment}
\usepackage{tikz}
\usepackage{hvfloat}
\usepackage{enumerate}
\usepackage{todonotes}
\usetikzlibrary{decorations.markings}
\usetikzlibrary{arrows}
\usepackage{graphicx}
\usepackage{caption}
\usepackage{subcaption}
\usepackage{mathtools}
\usepackage[all]{xy}

\usepackage{xcolor}

\makeatletter
\renewcommand{\boxed}[1]{\text{\fboxsep=.2em\fbox{\m@th$\displaystyle#1$}}}
\makeatother

\newcommand{\Z}{\mathbb{Z}}

\newcommand{\N}{\mathbb{N}}

\newcommand{\G}{\mathcal{G}}

\renewcommand{\le}{\leqslant}
\renewcommand{\ge}{\geqslant}

\theoremstyle{plain}
\newtheorem{thm}{Theorem}[section]
\newtheorem{lem}[thm]{Lemma}
\newtheorem{prop}[thm]{Proposition}

\newtheorem{theorem}{Theorem}

\newtheorem{corollary}{Corollary}

\makeatletter
\newtheorem*{rep@theorem}{\rep@title}
\newcommand{\newreptheorem}[2]{%
\newenvironment{rep#1}[1]{%
\def\rep@title{#2 \ref{##1}}%
\begin{rep@theorem}}%
{\end{rep@theorem}}}
\makeatother
\newreptheorem{theorem}{Theorem}

\newtheorem*{thm*}{Theorem}
\newtheorem*{lem*}{Lemma}
\newtheorem*{prop*}{Proposition}
\newtheorem*{cor*}{Corollary}
\newtheorem*{qu*}{Question}
\newtheorem*{dt*}{Definition and Theorem}
\newtheorem*{exmp*}{Example}
\newtheorem*{exmps*}{Examples}
\newtheorem*{dprop*}{Definition and Proposition}
\newtheorem*{conj*}{Conjecture}

\theoremstyle{definition}
\newtheorem{defn}[thm]{Definition}

\newtheorem*{defn*}{Definition}
\newtheorem*{not*}{Notation}

\theoremstyle{plain}

\newtheorem*{rem*}{Remark}

\DeclareMathOperator\Sym{Sym}

\begin{document}
\title{Invariable generation and wreath products}

\author{Charles Garnet Cox}
\address{School of Mathematics, University of Bristol, Bristol BS8 1UG, UK}
\email{charles.cox@bristol.ac.uk}

\thanks{}


\date{\today}
\begin{abstract}
Invariable generation is a topic that has predominantly been studied for finite groups. In 2014, Kantor, Lubotzky, and Shalev produced extensive tools for investigating invariable generation for infinite groups. Since their paper, various authors have investigated the property for particular infinite groups or families of infinite groups.

A group is invariably generated by a subset $S$ if replacing each element of $S$ with any of its conjugates still results in a generating set for $G$. In this paper we investigate how this property behaves with respect to wreath products. Our main work is to deal with the case where the base of $G\wr_X H$ is not invariably generated. We see both positive and negative results here depending on $H$ and its action on $X$.
\end{abstract}
\maketitle
\section{Introduction}
Invariable generation arises naturally in computational Galois theory, and has been actively studied in relation to many interesting questions. A group $G$ is invariably generated if there exists a set $S$ such that for any choice of elements $a_g \in G$, $g \in G$ we have $\langle a_s^{-1}sa_s\;:\;s \in S\rangle=G$. All finite groups are invariably generated, leading to a question of the size of the smallest set that invariably generates a given finite group. The usual citation is \cite{Dixon} from 1992, but \cite{intro2} in 1872 also considered this natural idea (as noted in \cite{topgen}). There are many exciting and unexpected results in this area, such as \cite[Thm. 1.3]{introfinite} which says that any non-abelian finite simple group can be invariably generated by two elements. The same authors also worked with invariable generation for infinite groups, developing a wide range of results in \cite{intro}. In the infinite case, there exist groups that are not invariably generated, for example any infinite group with just two conjugacy classes; uncountably many 2-generated torsion free examples of such groups were produced in \cite{osin2conjclass}. In \cite{Wiegold1, Wiegold2} the notion of groups where no proper subgroup meets every conjugacy class was considered, which is equivalent to the group being invariably generated; \cite{Wiegold1} showed that this property is closed under extensions, whereas \cite{Wiegold2} showed that it is not always preserved for subgroups. For infinite groups we can also make the distinction between groups that are invariably generated only by infinite sets, and finitely invariably generated: those for which a finite invariable generating set exists.

\begin{not*} We will write IG to denote that a group is invariably generated, FIG if it is finitely invariably generated, and $\neg$IG if the group is not invariably generated (or equivalently that the group itself is not an invariable generating set). Note that usage of IG throughout this paper will mean that the group is not FIG.
\end{not*}

\begin{defn*} Let $G$ and $H$ be non-trivial groups and let $X$ be a set on which $H$ acts faithfully. Then $G\wr_X H$ is the group $\left(\bigoplus_{x \in X} G_x\right)\rtimes H$, where elements of $H$ act, via conjugation, by multiplying the indices of $\left(\bigoplus_{x \in X} G_x\right)$ on the right. This copy of $H$ is called the \emph{head} of $G\wr_X H$. The \emph{base} of $G\wr_X H$ is the subgroup $\bigoplus_{x \in X} G_x$. We will write $G\wr H$, called a regular wreath product, for the specific case that $X=H$ (with the natural action of $H$ on itself by multiplication).
\end{defn*}
\begin{not*} Let $H_X\le \Sym(X)$ denote the group $H$ together with a faithful action of $H$ on $X$.
Also, for any $g \in G$ and $x \in X$, let $g^{(x)}$ denote the element $g$ in $G_x$.
\end{not*}
In this paper we systematically investigate how this construction behaves with respect to invariable generation. Current known results for infinite groups created from wreath products are \cite[Prop 2.11]{intro} which deals with iterated wreath products of finitely generated abelian groups and \cite{iteratedwreath} where a limit of an iterated wreath product of finite cyclic groups is considered. Other results in \cite{intro} do not add extensively to this. If both $G$ and $H$ are FIG or IG, then it follows from \cite[Cor. 2.3(iii)]{intro}, which states that the class of invariably generated groups is closed under extensions, that $G\wr_X H$ is IG or FIG. A result for the Frattini subgroup of an arbitrary wreath product appears difficult to produce, which prevents the use of \cite[Lem. 2.5]{intro}. Moreover if $X$ is infinite and $G$ is non-trivial, then the base of $G\wr_XH$ will not be finitely generated and so cannot be FIG. This makes the other tools in \cite{intro} mostly unusable for studying wreath products. Our approach is more combinatorial in nature and involves either showing that a general generating set cannot invariably generate our group, or showing that any conjugates of a particular generating set will generate the group. 
\begin{not*} Given a group $G$ and elements $a, b \in G$, let $a\sim b$ denote that $a$ is conjugate to $b$ in $G$ i.e.\ that there exists an $x\in G$ such that $x^{-1}ax=b$.
\end{not*}

If $H$ is $\neg$IG, then a simple argument by contradiction shows that $G\wr_X H$ is $\neg$IG (after noting that given $k, k' \in H$ with $k\sim k'$, then for any $w \in\bigoplus_{x \in X} G_h$ there exists a $w' \in \bigoplus_{x \in X} G_h$ such that $wk\sim w'k'$). Similarly if $H$ is IG then $G\wr_X H$ cannot be FIG. In Proposition \ref{firstargument}, we show that if $G$ and $H$ are FIG and $G\wr_XH$ is finitely generated, then $G\wr_X H$ is FIG. These  results appear in Section 2, and are summarised in the following table.
\begin{table}[h!]
\begin{tabular}{lcccc}
                    & \multicolumn{1}{l}{}          & \multicolumn{3}{c}{H}      \\
                    & \multicolumn{1}{c|}{}         & FIG    & IG     & $\neg$IG \\ \cline{2-5}
\multicolumn{1}{c}{} & \multicolumn{1}{c|}{FIG}      & FIG    & IG     & $\neg$IG \\
G                    & \multicolumn{1}{c|}{IG}       & - & IG     & $\neg$IG \\
                    & \multicolumn{1}{c|}{$\neg$IG} & - & - & $\neg$IG
\end{tabular}
\end{table}

In order to simplify the description of our results, we create a new definition. The idea is to generalise, to $G\wr_XH$, the situation where the head of $G\wr H$ is torsion.
\begin{defn*} We say that $H_X$ is of torsion-type if there exists a $y \in X$ such that for all $x \in yH_X$ and all $k \in H_X$, $x\langle k\rangle$ is finite.
\end{defn*}
The remaining cases depend entirely on whether $G\wr_X H$ is finitely generated (since this is an obstruction to $G\wr_X H$ being FIG) and whether $H_X$ is of torsion-type. We prove Theorem \ref{mainthmA} in Section 3 and Theorem \ref{mainthmB} in Section 4.

\begin{theorem} \label{mainthmA} Let $H_X$ be of torsion-type.
\begin{enumerate}[i)]
\item If $H$ is FIG and $G$ is IG, then $G\wr_X H$ is IG.
\item If $H$ is FIG or IG and $G$ is $\neg$IG, then $G\wr_X H$ is $\neg$IG.
\end{enumerate}
\end{theorem}
This theorem was partly investigated because of \cite[Problems 1 and 2]{intro} which ask whether taking a finite index subgroup preserves the properties of IG and FIG, respectively. Theorem \ref{mainthmA} states that wreath products $G\wr_XH$ with $X$ and $H$ finite cannot provide examples where $G$ is IG and $G\wr_XH$ is FIG or where $G$ is $\neg$IG and $G\wr_XH$ is FIG or IG.

The other case is perhaps more surprising. Given $H_X$ that is not of torsion-type, essentially it says that the only impact that $G$  can have on whether $G\wr_XH$ is FIG, IG, or $\neg$IG, is for $G$ and $X$ to be chosen so that $G\wr_XH$ is not finitely generated and so not FIG. The proof involves fixing  elements $t\in H$ of infinite order and then showing that powers of any conjugate of $g^{(x)}t$ allow us to retrieve, in some sense, the element $g$. This occurs because conjugacy in $G\wr_X H$ under these hypotheses behaves more like multiplication in the group.

\begin{theorem} \label{mainthmB} Let $H_X$ be not of torsion-type.
\begin{enumerate}[i)]
\item If $G\wr_XH$ is finitely generated and $H$ is FIG, then $G\wr_X H$ is FIG.
\item If $G\wr_XH$ is not finitely generated and $H$ is FIG, then $G\wr_X H$ is IG.
\item If $H$ is IG, then $G\wr_X H$ is IG.
\end{enumerate}
\end{theorem}
Theorem \ref{mainthmB} provides a natural embeddability result. The author is unaware of a finitely generated IG group being explicitly stated in the literature.
\begin{qu*} Does there exist a finitely generated IG group (that is not FIG)?
\end{qu*}
\noindent{\bf Edit:} A positive answer to this question has now been given independently by Minasyan and Goffer-Lazarovich in \cite{answered1} and \cite{answered2} respectively. From their results, the assumption that such a group exists can be removed from the following corollary. Corollary A then yields a continuum of examples of finitely generated IG groups.
\begin{corollary} Let $G$ be a finitely generated group and $H$ a group that is not finitely generated (so that it may be of arbitrarily large cardinality). Assume there exists a group $A$ that is IG and finitely generated. We then have that:
\begin{itemize}
\item $G\wr\Z$ is FIG;
\item $G\wr (A\times \Z)$ is IG and finitely generated, and moreover $G\wr A$ is IG and finitely generated if $A$ is not torsion;
\item $H\wr\Z$ is IG.
\end{itemize}
Hence every finitely generated group embeds into a FIG group and every group embeds into an IG group. Should a group exist that is IG and finitely generated, then every finitely generated group embeds into a finitely generated IG group.
\end{corollary}
\begin{proof} All of the statements follow immediately from Theorem \ref{mainthmB}.
\end{proof}
Note the stark contrast between this corollary and \cite[Thm. 1.1]{osin2conjclass} which states that any countable group can be embedded into one where all elements of the same order are conjugate. For torsion-free groups, this is in some sense the furthest that we can be from being invariably generated. We now turn our attention to iterated wreath products.
The following relies on the above table and Theorems A and B.
\begin{corollary} Let $G^{(n)}:=(\ldots((G_1\wr_{X_2} G_2)\ldots)\wr_{X_{n-1}} G_{n-1})\wr_{X_{n}} G_n$. If $G_2, \ldots, G_n$ with respect to $X_2, \ldots, X_n$ are all of torsion-type, then set $k:=1$. Otherwise, let $k$ be the largest number in $\{1, \ldots, n\}$ such that $(G_k)_{X_k}$ is not of torsion-type. Then
\begin{enumerate}[i)]
\item $G^{(n)}$ is FIG if, and only if,  $G^{(n)}$ is finitely generated and $G_k, \ldots, G_n$ are FIG.
\item $G^{(n)}$ is IG if, and only if, $G_k, \ldots, G_n$ are IG or FIG and (i) does not occur.
\end{enumerate}
In the case of an iterated regular wreath product, (i) becomes that $G_1, \ldots, G_{k-1}$ are finitely generated and $G_k, \ldots, G_n$ are FIG.
\end{corollary}
\begin{proof} If $G_2, \ldots, G_n$ are all of torsion-type (with respect to $X_2, \ldots, X_n$), then the result follows from either repeatedly applying Theorem \ref{mainthmA}(i) or repeatedly applying Proposition \ref{firstargument}.

For any $m \in \{1, \ldots, n\}$, let $G^{(m)}:=(\ldots((G_1\wr_{X_2} G_2)\ldots)\wr_{X_{m-1}} G_{m-1})\wr_{X_m} G_m$. We have that $G_k$ is not of torsion-type whereas $G_{k+1}, \ldots, G_n$ are of torsion-type. If $G_k$ is $\neg$IG, then $G^{(k)}$ is $\neg$IG by Lemma \ref{notIG}. Then $G^{(k+1)}, \ldots, G^{(n)}$ are all also $\neg$IG by Theorem \ref{mainthmA}(ii). If $G_k$ is IG, then $G^{(k)}$ is IG by Theorem \ref{mainthmB}(iii). Then $G^{(k+1)}$ is IG if and only if $G_{k+1}$ is either IG or FIG by Theorem \ref{mainthmA}(i) and Lemma \ref{notIG}. If $G_k$ is FIG, then $G^{(k)}$ is FIG if and only if it is finitely generated (again by Theorem \ref{mainthmB}). Then $G^{(k+1)}$ is FIG if and only if $G_{k+1}$ is FIG by Lemma \ref{notFIG}, Lemma \ref{notIG}, and Proposition \ref{firstargument}. If $G^{(k+1)}$ is IG or $\neg$IG, then Theorem \ref{mainthmA} states that $G^{(k+2)}, \ldots, G^{(n)}$ are also either IG or $\neg$IG. Hence $G^{(n)}$ can only be FIG if $G^{(k)}, \ldots, G^{(n)}$ are all FIG.
\end{proof}
Our arguments also apply to profinite groups where the notion of topological generation can be used to produce notions analogous to FIG, IG, and $\neg$IG. For topological generation we do have a finitely generated IG group, for example in \cite{qfortopgen} and \cite[Section 4]{intro}.

\vspace{0.3cm}
\noindent\textbf{Acknowledgements.} I thank Gareth Tracey from the University of Bath for introducing me to the topic of invariable generation. I thank Tim Burness at the University of Bristol for his advice and encouragement. Finally, I thank the referee for their clear and helpful comments.

\section{Initial observations}
We first compute the conjugates of the base and the head of $G\wr_X H$. In order to describe the form of these conjugates, the following definition is useful.
\begin{defn*} Given $K\le G$ and $g \in G$, a $K$-conjugate of $g$ is any element in $\{k^{-1}gk\;:\;k \in K\}$.
\end{defn*}
\begin{lem} \label{firstcalcs} Let $w \in \bigoplus_{x \in X}G_x$, $k \in H$, and $y \in X$. If $g \in G_y$, then we have that $(wk)^{-1}g(wk)\in G_{yk}$. If $h \in H$, then any $G\wr H$-conjugate of $h$ decomposes as the product of an $H$-conjugate of $h$ and an element of $\bigoplus_{x \in X}G_x$.
\end{lem}
\begin{proof}
Fix a $y \in X$ and let $g \in G_y$. Let $X_y:=X\setminus\{y\}$. Given any $k \in H$ and $w \in \bigoplus_{x \in X}G_x$, we have that $w=\bigoplus_{x \in X}w^{(x)}$. Then
\begin{equation*}
\begin{split}
(wk)^{-1}g(wk)&=k^{-1}(w^{-1}gw)k\\
&=k^{-1}\Bigg(\bigoplus_{x \in X_y}(w^{(x)})^{-1}((w^{(y)})^{-1}gw^{(y)})\bigoplus_{x \in X_y}w^{(x)}\Bigg)k\\
&=k^{-1}((w^{(y)})^{-1}gw^{(y)})k
\end{split}
\end{equation*}
which is of the form $k^{-1}f^{(y)}k$ for $f=w^{-1}gw$, a $G$-conjugate of $g$. Notice that $k^{-1}f^{(y)}k=f^{(yk)}\in G_{yk}$. We now conjugate $h \in H$ by $(wk)^{-1}$, and see that
\begin{equation*}
(wk)h(wk)^{-1}=w(khk^{-1})w^{-1}=wh'w^{-1}=wuh'
\end{equation*}
where $h'=khk^{-1}$, a $H$-conjugate of $h$, and $wu \in \bigoplus_{x \in X}G_x$.
\end{proof}

\begin{lem}\label{notFIG} Let $H$ be IG. Then $G\wr_X H$ cannot be FIG.
\end{lem}
\begin{proof} Since $H$ is not FIG, given any finite set $\{h_1, \ldots, h_m\}$ there exist $a_1, \ldots, a_m$ such that $\langle a_1^{-1}h_1a_1, \ldots a_m^{-1}h_ma_m\rangle \ne H$. Hence given $\{w_1h_1, \ldots, w_mh_m\}\subset G\wr H$, we have that $\{a_1^{-1}w_1h_1a_1, \ldots, a_m^{-1}w_mh_ma_m\}$ does not generate $H$.
\end{proof}
From this lemma and the fact that invariable generation is extension closed, if $G$ is FIG or IG and $H$ is IG, then $G\wr_X H$ is IG.
\begin{lem}\label{notIG} Let $H$ be $\neg$IG. Then $G\wr_X H$ is $\neg$IG.
\end{lem}
\begin{proof} This follows immediately from Lemma \ref{firstcalcs}. Since $H$ is $\neg$IG, there exist $a_h \in H$ such that $\{a_h^{-1}ha_h\;:\;h \in H\}$ does not generate $H$. Then $\{wh\;:\;h \in H, w \in \bigoplus_{x \in X}G_x\}$ cannot be an invariable generating set for $G\wr_X H$ since using the $a_h$ defined above we have that $\langle a_h^{-1}wha_h\;:\;h \in H, w \in \bigoplus_{x \in X}G_x\rangle$ cannot contain $H$ (and so does not generate $G\wr_X H$).
\end{proof}
The author is unaware of a source for the following, but expects it is well known.
\begin{lem} \label{finitelygen}
Let $G\wr_XH$ be finitely generated. Then $G$ and $H$ are finitely generated. Moreover, there exist $y_1, \ldots, y_d \in X$ such that $\bigcup_{i=1}^dy_iH_X=X$.
\end{lem}
\begin{proof}
Assume $G\wr_XH$ is finitely generated. Then $\langle w_1h_1, \ldots, w_nh_n\rangle=G\wr_XH$ where for every $i \in \{1, \ldots, n\}$ we have $w_i \in \bigoplus_{x \in X}G_x$ and $h_i \in H$. Now
\begin{equation*}
\langle w_1h_1, \ldots, w_nh_n\rangle \le \langle w_1, \ldots, w_n, h_1, \ldots, h_n\rangle
\end{equation*}
which implies that $H$ is finitely generated. For each $i$, we can choose $k_i \in \N$, $u_{i, 1}, \ldots, u_{i, k_i} \in G$, and $x_{i, 1}, \ldots, x_{i, k_i} \in X$ such that $w_i=u_{i, 1}^{(x_{i, 1})}\ldots u_{i, k_i}^{(x_{i, k_i})}$. Thus
\begin{equation*}
\langle w_1, \ldots, w_n, h_1, \ldots, h_n\rangle \le \langle \{u_{i, j}^{(x_{i, j})}\;:\;1\le i\le n\text{ and }1\le j\le k_i\}\cup H_X\rangle.
\end{equation*}
Our assumption that this is equal to $G\wr_X H$ now implies that any $x \in X$ lies in $x_{i, j}H$ for some $i, j \in \N$ and that $\langle u_{i, j}\;:\;1\le i\le n$ and $1\le j\le k_i\rangle=G$.
\end{proof}

The following lemma provides a sufficient condition for a subset of $G\wr_XH$ to generate $G\wr_XH$.

\begin{lem}\label{GHgenerate} Let $\{y_i\;:\;i \in I\}\subseteq X$ have the property that $X=\bigcup_{i \in I}y_iH$, let $S_H$ invariably generate $H$, and $S_H':=\{a_s^{-1}sa_s\;:\;s \in S_H\}$ for some - completely free - choice of $\{a_s\;:\;s \in S_H\}\subseteq G\wr_X H$. Then $\langle \bigcup_{i \in I}G_{y_i} \cup S_H'\rangle =G\wr_X H$.
\end{lem}
\begin{proof} From Lemma \ref{firstcalcs}, $S_H'=\{w_sh_s\;:\;s \in S_H\}$ where $w_s \in \bigoplus_{x \in X}G_x$, $h_s \in H$ for every $s \in S_H$, and every $h_s$ is $H$-conjugate to an element in $S_H$. Moreover $\langle h_s\;:\;s \in S_H\rangle=H$ from our hypothesis that $S_H$ invariably generates $H$. Hence, given any $k \in H$, there exists an element $wk \in\langle S_H'\rangle$ where $w \in \bigoplus_{x \in X}G_x$. Then $(wk)^{-1}G_x(wk)=G_{xk}$ and so $\bigoplus_{x \in X}G_x\le \langle \bigcup_{i \in I}G_y \cup S_H'\rangle$. Therefore, for every $s \in S_H$, we have that $w_s \in \langle \bigcup_{i \in I}G_y \cup S_H'\rangle$ and so this set also contains $\{h_s\;:\;s \in S_H\}$.
\end{proof}
Our aim is now to take a generating set $S$ made from invariable generating sets in $G$ and $H$ and show that  this set $S$ invariably generates $G\wr_XH$. We do this by showing that replacing the elements of $S$ with any conjugates results in a set that satisfies the sufficient condition introduced in the previous lemma.
\begin{prop}\label{firstargument} Let $G$ and $H$ be FIG. Then $G\wr_X H$ is either FIG or IG. Moreover, $G\wr_X H$ is FIG if and only if it is finitely generated. In particular, $G\wr H$ is FIG.
\end{prop}
\begin{proof} Let $Y:=\{y_i\;:\;i \in I\}\subseteq X$ have the property that $X=\bigcup_{y \in Y}yH$, and let $\{g_1, \ldots, g_n\}$ and $\{h_1, \ldots, h_m\}$ be finite invariable generating sets for $G$ and $H$. From Lemma \ref{finitelygen}, if $G\wr_XH$ is finitely generated, then $Y$ can be chosen to be finite. Our claim is that $\bigcup_{i \in I}\{g^{(y_i)}_1, \ldots, g^{(y_i)}_n\}\cup \{h_1, \ldots, h_m\}$ is an invariable generating set for $G\wr_X H$. From Lemma \ref{firstcalcs}, any $g_j^{(y_i)}$ can only be conjugate to an element of the form $f_{i,j}^{(x_{i,j})} \in G_{x_{i,j}}$ where $x_{i,j} \in X$ and $f_{i,j}^{(x_{i,j})}$ is $G_{x_{i,j}}$-conjugate to $g_j^{(x_{i,j})}$. Also each $h_i$ is conjugate to an element $w_ih_i'$ where $w_i \in \bigoplus_{x \in X}G_x$, $h_i'\in H$, and $h_i'$ is $H$-conjugate to $h_i$.

Note that $\langle h_1', \ldots, h_m'\rangle=H$ from our assumption that $\{h_1, \ldots, h_m\}$  invariably generates $H$. So, for each $x_{i,j} \in X$ there exists $k_{i,j} \in H$ such that $x_{i,j}k_{i,j}=y_i$ and there exists $u_{i,j} \in \bigoplus_{x \in X}G_x$ such that $u_{i,j}k_{i,j} \in \langle w_1h_1', \ldots, w_mh_m'\rangle$. Now, for each $i \in I$ and $j \in \{1, \ldots, n\}$, let
\begin{equation*}
a_{i,j}:=(u_{i,j}k_{i,j})^{-1}f_i^{(x_{i,j})}(u_{i,j}k_{i,j}).
\end{equation*}
Note that $a_{i,j} \in G_{y_i}$ for every $i$. Moreover, since $\{g_1, \ldots, g_n\}$ invariably generates $G$, we have that $\langle a_{i, 1}, \ldots, a_{i, n}\rangle =G_{y_i}$ for every $i \in I$. Lemma \ref{GHgenerate} yields the result.
\end{proof}
The above argument also works if $\{g_1, \ldots, g_n\}$ and $\{h_1, \ldots, h_m\}$ are replaced with any invariable generating sets for $G$ and $H$.
We finish this section with the final result that does not depend on whether or not $H_X$ is of torsion-type.

%
%
\begin{lem} If $G$ is IG, $H$ is FIG, and $G\wr_X H$ is not finitely generated, then $G\wr_XH$ is IG.
\end{lem}
\begin{proof} Invariable generation is extension closed.
\end{proof}

\section{The case where $H_X$ is of torsion-type}
The notion of torsion-type generalises the case where the head of $G\wr H$ is torsion. Note that if $H$ acts transitively on $X$, then the definition becomes much simpler.
\begin{defn*} Let $H_X\le \Sym(X)$ denote the group $H$ together with a faithful action of $H$ on $X$. Then $H_X$ is of torsion-type if there exists a $y \in X$ such that for all $x \in yH_X$ and all $k \in H_X$, $x\langle k\rangle$ is finite.
\end{defn*}
For any $k \in H$ and $w = \bigoplus_{x \in X}w_x^{(x)}$ where each $w_x^{(x)}\in G_x$, we will show that the torsion-type hypothesis on $H_X$ means that $wk$ is conjugate to an element roughly of the form $\bigoplus_{x \in X}(a_x^{-1}w_xa_x)^{(x)}k$ where each $a_x \in G_x$. We do this in two stages.
\begin{lem} \label{cosetreps} Let $H_X$ be of torsion-type, $y \in X$, $k \in H$, $C:=y\langle k\rangle$, $u \in \bigoplus_{x \in C}G_x$, and $v \in \bigoplus_{x \in X\setminus C}G_x$. Then $uvk$ is conjugate to $u_*vk$, where $u_* \in G_y$.
\end{lem}
\begin{proof} Let $|y\langle k\rangle|=d+1$ and $u=u_0^{(y)}u_1^{(yk)}\ldots u_d^{(yk^d)}$ where $u_0, \ldots, u_d \in G$. Then
\begin{align*}
uvk&=vu_0^{(y)}u_1^{(yk)}\ldots u_d^{(yk^d)}k\\
&\sim (u_d^{-1})^{(yk^d)}vu_0^{(y)}u_1^{(yk)}\ldots u_d^{(yk^d)}ku_d^{(yk^d)}\\
&=vu_0^{(y)}u_1^{(yk)}\ldots u_{d-1}^{(yk^{d-1})}\big(ku_d^{(yk^{d})}k^{-1}\big)k\\
&=vu_0^{(y)}u_1^{(yk)}\ldots u_{d-1}^{(yk^{d-1})}u_{d}^{(yk^{d-1})}k\\
&\sim\ldots\\
&=vu_*^{(y)}k.\qedhere
\end{align*}
\end{proof}
The second stage introduces the elements, in $G$, that we wish to conjugate by.
\begin{lem} \label{conjworks} Let $H_X$ be of torsion-type, $y \in X$, $k \in H$, and $C:=y\langle k\rangle$. Then $u^{(y)}\bigoplus_{x \in X\setminus C}v_x^{(x)}k\sim(g^{-1}ug)^{(y)}\bigoplus_{x \in X\setminus C}v_x^{(x)}k$ for every $g \in G$.
\end{lem}
\begin{proof} Let $|y\langle k\rangle|=d+1$. For any $g \in G$, we can conjugate $u^{(y)}\bigoplus_{x \in X\setminus C}v_x^{(x)}k$ by $g^{(y)}g^{(yk)}\ldots g^{(yk^d)}$. After much cancelling, we obtain the result.
\end{proof}

We can now prove Theorem \ref{mainthmA}, which finishes the classification of $G\wr_X H$ for when $H_X$ is of torsion-type.

\begin{proof}[Proof of Theorem \ref{mainthmA}] We assume that $H_X$ is of torsion-type. Our aims are the following:
\begin{enumerate}[i)]
\item If $H_X$ is FIG, and $G$ is IG, then $G\wr_X H$ is IG;
\item If $H_X$ is FIG or IG, and $G$ is $\neg$IG, then $G\wr_X H$ is $\neg$IG.
\end{enumerate}

Let $y \in X$ be chosen so that $Y:=yH_X$ is a set such that $x\langle k\rangle$ is finite for all $x \in Y$ and $k \in H_X$. Such a $y$ exists due to our assumption that $H_X$ is torsion-type.

Let $w = \bigoplus_{x \in X}w_x^{(x)}$, which we can write as $\bigoplus_{x \in Y}w_x^{(x)}v$ for some $v \in \bigoplus_{x \in X\setminus Y}G_x$. Then the set of conjugates of $w$ contains $\bigoplus_{x \in Y}(a_x^{-1}w_xa_x)^{(x)}v$ for any choice of $\{a_x \in G\;:\;x \in Y\}$.

Next consider $wk$ with $w$ as above and $k \in H$. Then $w = \big(\bigoplus_{x \in F}w_x\big)v$ for some finite set $F\subseteq Y$ and some $v \in \bigoplus_{x \in X\setminus Y}G_x$. Note that
$$F\subseteq x_1\langle k\rangle \cup x_2\langle k\rangle\cup \ldots x_n\langle k\rangle$$
for some $x_1, \ldots, x_n \in X$, where $x_i\langle k\rangle\cap x_j\langle k\rangle= \emptyset$ for all $i\ne j$.

By applying Lemma \ref{cosetreps} and then repeatedly applying Lemma \ref{conjworks}, we have
\begin{align*}
wk=\Big(\bigoplus_{x \in F}w_x^{(x)}\Big)vk\sim \Big(\bigoplus_{i=1}^nu_i^{(x_i)}\Big)vk\sim \Big(\bigoplus_{i=1}^n(a_i^{-1}u_ia_i)^{(x_i)}\Big)vk
\end{align*}
where we have complete choice over $a_1, \ldots, a_n \in G$.

Hence, both when $k=e_H$ and $k\ne e_H$, $wk$ is conjugate to an element with each component in $Y$ either trivial or of the form $a^{-1}ua$ where $a$ is free to choose.

We will now show that $S=\{w_jh_j\;:\;w_j \in \bigoplus_{x \in X}G_x, h_j \in H\}$ cannot invariably generate $G\wr_X H$ (assuming that $S$ is finite to prove (i) and assuming no restriction on $S$ to prove (ii)). We may conjugate each $w_jh_j \in S$ to an element of the form $(a_{j, 1}^{-1}w_{j, 1}a_{j, 1})^{(y_{j, 1})}\ldots (a_{j, k_j}^{-1}w_{j, k_j}a_{j, k_j})^{(y_{j, k_j})}v_jh_j$ where $k_j \in \N$, $v_j \in\bigoplus_{x \in X\setminus Y}G_x$, $a_{j, l}, w_{j, l} \in G_{y_{j, l}}$ for some $y_{j, l}\in Y$, and we are free to choose each $a_{j,l}$. For each $j \in S$ and $1\le l\le k_j$ let $w_{j,l}':=(a_{j, l}^{-1}w_{j, l}a_{j, l})^{(y_{j, l})}$. With a philosophy similar to that in the proof of Lemma \ref{finitelygen}, we have
\begin{equation*}
\begin{split}
\langle w_{j, 1}'\ldots w_{j, k_j}'v_jh_j\;:\;j \in S\rangle &\le \langle \{w_{j, 1}'\ldots w_{j, k_j}'v_j\;:\;j \in S\}\cup H_X\rangle\\
&\le \langle \bigcup_{x \in X\setminus Y}\!\!\!\!G_x\cup \{w_{j, l}'\;:\;j \in S, 1\le l\le k_j\}\cup H_X\rangle
\end{split}
\end{equation*}
which can only generate $G\wr_X H$ if $A=\{(a_{j, l}^{-1}w_{j, l}a_{j, l})^{(y)}\in G_y\;:\;j \in S, 1\le l\le k_j\}$ generates $G_y$. If $G$ is $\neg$IG, then we can choose the elements $a_{j, l}$ so that this is not the case, implying that $G\wr_X H$ is $\neg$IG. For (i), we may assume that no finite set invariably generates $G$. Thus, since any finite choice of $S$ results in $A$ being finite, there exists a choice of $\{a_{j, l}\;:\;j \in S, 1\le l\le k_j\}$ such that $\langle A\rangle\ne G$. Hence if $G$ is IG, then $G\wr_X H$ is not FIG, and so must be IG.
\end{proof}
In our final section we deal with the case where the above argument cannot be applied i.e.\ where $H_X$ is not of torsion-type.

\section{The case where $H_X$ is not of torsion-type}

The following proposition introduces our invariable generating set for $G\wr_XH$, under the assumption that $H_X$ is not of torsion-type.

\begin{prop} \label{nottorsionprop}
Let $H_X$ be FIG or IG and not of torsion-type. Let $\{y_i\in X\;:\;i \in I\}$ have the properties that $X=\bigcup_{i \in I}y_iH$ and that there exist $\{t_i\in H_X\;:\;i \in I\}$ with $y_i\langle t_i\rangle$ infinite for every $i \in I$. Now, for any generating sets $\G_i$ of $G_{y_i}$ and an invariable generating set $S_H$ of $H$, we have that $S_H\cup\big(\bigcup_{i \in I}\G_i\cup\G_it_i\cup\{t_i\}\big)$ invariably generates $G\wr_XH$.
\end{prop}

It is worth remarking that if $H_X$ is not of torsion-type, then there is a set $\{y_i\in X\;:\;i \in I\}$ which satisfies both of the conditions in the above proposition.

\begin{proof}[Proof of Theorem \ref{mainthmB} assuming Proposition \ref{nottorsionprop}] From Lemma \ref{notFIG}, if $H$ is IG, then $G\wr_X H$ is not FIG. Similarly if $G\wr_XH$ is not finitely generated, then it cannot be FIG. Proposition \ref{nottorsionprop} then provides an infinite invariable generating set for $G\wr_X H$. From Lemma \ref{finitelygen}, $G\wr_XH$ is finitely generated if and only if $G$ and $H$ are finitely generated and there exist $y_1, \ldots, y_n \in X$ such that $\bigcup_{i=1}^ny_iH_X=X$. With $H_X$ not of torsion-type, we can choose $y_1,\ldots y_n\in X$ such that there exist $t_1, \ldots, t_n \in H_X$ with $y_i\langle t_i\rangle$ infinite for each $i \in \{1, \ldots, n\}$. Hence in the case where $G\wr_X H$ is finitely generated and $H$ is FIG, we can apply Proposition \ref{nottorsionprop} to conclude that $S_{y_1, t_1}\cup \ldots \cup S_{y_n, t_n}$ is a finite set that invariably generates $G\wr_XH$.
\end{proof}

Thus all that is required is to prove the proposition. We now fix a choice of conjugators. Note that this notation does not introduce any restrictions to our results, and is merely for convenience.

\begin{not*}
Fix some choice of conjugators $a_{wk}\in G\wr_X H$ for every $wk \in G\wr_XH$. For any set $S\subseteq G\wr_XH$, let $S':=\{a_s^{-1}sa_s\;:\;s \in S\}$.
\end{not*}
\begin{defn} Let $z \in X$. Then a $\Gamma_z$-set is a subset of $\bigoplus_{x \in X}G_x$ such that for every $g \in G$ it contains an element $\gamma_g= \bigoplus_{x \in X}g_x^{(x)}$ where $g_x^{(x)} \in G_x$ for each $x \in X$ and $g_z^{(z)}=g$.
\end{defn}
We start by explaining the reason for wanting $\Gamma_z$-sets, which is that they will be a stepping stone towards a generating set.
\begin{lem}\label{thestrategy} Let $S_H$ invariably generate $H$.  Let $\{y_i\;:\;i \in I\}\subseteq X$ have the property that $X=\bigcup_{i \in I}y_iH$. For each $i \in I$, let $\G_i$ generate $G_{y_i}$, let $z_i \in y_iH$, and let $\Gamma_i$ be a $\Gamma_{z_i}$-set. Then $\langle S_H'\cup\big(\bigcup_{i \in I} \G_i'\cup\Gamma_i\big)\rangle=G\wr_X H$.
\end{lem}
\begin{proof} Fix an $i \in I$ and let $g \in \G_i$. By Lemma \ref{firstcalcs}, any conjugate of $g$ has the form $a^{-1}g^{(x)}a$ where $x \in y_iH$ and $a \in G_x$. We can then conjugate $a^{-1}g^{(x)}a$ to $b^{-1}g^{(z_i)}b$, for some $b \in G_{z_i}$, using an element of $\langle S_H'\rangle$. Thus $b^{-1}g^{(z_i)}b\in \langle S_H'\cup \G_i'\rangle$. By assumption $\gamma_b \in \Gamma_i$, and so
\begin{equation*}
\gamma_bb^{-1}g^{(z_i)}b\gamma_b^{-1}=g^{(z_i)}\in \langle S_H'\cup \G_i'\cup \Gamma_i\rangle.
\end{equation*}
Now $G_{z_i}\le \langle S_H'\cup \G_i'\cup \Gamma_i\rangle$, since $g$ was arbitrary. Moreover, our argument applies to every $i \in I$, and so
\begin{equation*}
\langle S_H'\cup\big(\bigcup_{i \in I} \G_i'\cup\Gamma_i\big)\rangle\ge \langle \bigcup_{i \in I}G_{z_i} \cup S_H'\rangle
\end{equation*}
which equals $G\wr_X H$ by Lemma \ref{GHgenerate}.
\end{proof}
\begin{not*} Fix:
\begin{itemize}
\item[i)] a $y\in X$ and $t\in H_X$ such that $y\langle t\rangle$ is infinite;
\item[ii)] $S_H$ to be an invariable generating set of $H$; and
\item[iii)] $\G$ to be a generating set of $G_y$.
\end{itemize}
\end{not*}
Our remaining computations will show that $\langle S_H'\cup(\G t)'\cup\{t\}'\rangle$ contains a $\Gamma_z$-set for some $z\in yH_X$. The form of $\{t\}'$ and elements in $S_H'$ and $\G'$ are known from Lemma \ref{firstcalcs}.

\begin{lem}\label{lemelementsalphag} Let $g \in \G$. Then $\langle (\G t)'\cup \{t\}'\cup S_H' \rangle$ contains an element of the form (\ref{elementsalphag}), where $a_x^{(x)}\in G_x$ for each $x \in X$.
\begin{equation}\label{elementsalphag}
\Big(\bigoplus_{x \in X}a_x^{(x)}\Big)^{-1}g^{(y)}t\bigoplus_{x \in X}a_x^{(x)}
\end{equation}
\end{lem}
\begin{proof} By definition $g^{(y)}t \in \G t$. Any conjugate of $g^{(y)}t$ has the form
\begin{equation*}
(ak)^{-1}g^{(y)}t(ak)=k^{-1}(a^{-1}g^{(y)}ta)k
\end{equation*}
where $a \in \bigoplus_{x \in X}G_x$ and $k \in H$. But $k\in \langle S_H\rangle$ and so, from our assumption that $S_H$ is an invariable generating set of $H$ and Lemma \ref{firstcalcs}, $\langle S_H'\rangle$ contains an element of the form $wk$ for some $w \in \bigoplus_{x \in X}G_x$. Computing,
\begin{equation*}
(wk)k^{-1}(a^{-1}g^{(y)}ta)k(wk)^{-1}=wa^{-1}g^{(y)}taw^{-1}=(aw^{-1})^{-1}g^{(y)}t(aw^{-1}).\qedhere
\end{equation*}
\end{proof}

Our key observation is that, for any $g \in G$, taking powers of $\{g^{(y)}t\}'$ results in an element with some $G_x$ components equal to $g$, in a computable and controlled way. This will allow us to produce a $\Gamma_{z}$-set for some known $z \in yH_X$. Under our assumption that $H_X$ is not of torsion-type and that $t \in H_X$ is an element with an infinite orbit, we will reduce to the simpler case of $\bigoplus_{x \in X}G_x\rtimes\langle t\rangle$ which behaves like $G\wr \Z$. The following notation reflects our wish to consider this simplification.
\begin{not*}
For any $g\in G$ and $i \in \Z$, let $g^\boxed{i}$ denote the element $g$ in $G_{yt^i}$.
\end{not*}
In light of Lemma \ref{lemelementsalphag}, we introduce some notation.
\begin{not*} For each $g\in\G$, let $\alpha_g$ denote a choice of an element in $\langle S_H'\cup (\G t)'\cup \{t\}'\rangle$ of the form (\ref{elementsalphag}).
\end{not*}
To aid our calculations, we fix some notation for two elements in $\{\alpha_g\;:\;g\in\G\}$.
\begin{not*} Let $e$ denote the identity element of $G$ and fix an $f\in\G$. Then
\begin{equation*}
\alpha_e:=\bigoplus_{i<|c|}A_i^\boxed{i}t\bigoplus_{i<|c|}a_i^\boxed{i}\sigma_e\text{ and } \alpha_f:=\bigoplus_{i<|d|}B_i^\boxed{i}f^\boxed{0}t\bigoplus_{i<|d|}b_i^\boxed{i}\sigma_f
\end{equation*}
where $c, d \in \N$ and for every $i \in \Z$ we have
\begin{itemize}
\item $a_i^\boxed{i}, b_i^\boxed{i}\in G_{yt^i}$, and $f^\boxed{0}\in G_y$
\item $A_i^\boxed{i}:=\Big(a_i^\boxed{i}\Big)^{-1}$ and $B_i^\boxed{i}:=\Big(b_i^\boxed{i}\Big)^{-1}$
\item $\sigma_e, \sigma_f \in \bigoplus_{x \in X\setminus y\langle t\rangle}G_x$.
\end{itemize}
\end{not*}

\begin{lem} \label{alphamlem} Let $m \in \N$. Then the element $\alpha_e^{-m}\alpha_f^m$ is of the form
\begin{equation}\label{alpham}
\bigoplus_{i<|c|}A_i^\boxed{i}\bigoplus_{i<|c|}a_{i}^\boxed{i+m}\bigoplus_{i<|d|}B_{i}^\boxed{i+m}f^\boxed{1}f^\boxed{2}\ldots f^\boxed{m}\bigoplus_{i<|d|}b_i^\boxed{i}\sigma
\end{equation}
where $\sigma \in \bigoplus_{x \in X\setminus y\langle t\rangle}G_x$.
\end{lem}
\begin{proof} Since $t$ conjugates $\bigoplus_{x \in X\setminus y\langle t\rangle}G_x$ to itself, we have $(t^{-1}\sigma_e)^m(t\sigma_f)^m=\sigma$ for some $\sigma\in\bigoplus_{x \in X\setminus y\langle t\rangle}G_x$. Computing,
\begin{equation*}
\begin{split}
\alpha_e^{-m}\alpha_f^m&=\Big(\bigoplus_{i<|c|}A_i^\boxed{i}t\bigoplus_{i<|c|}a_i^\boxed{i}\sigma_e\Big)^{-m}\Big(\bigoplus_{i<|d|}B_i^\boxed{i}f^\boxed{0}t\bigoplus_{i<|d|}b_i^\boxed{i}\sigma_f\Big)^m\\
&=\Big(\bigoplus_{i<|c|}A_i^\boxed{i}t^{-1}\bigoplus_{i<|c|}a_i^\boxed{i}\sigma_e\Big)^m\Big(\bigoplus_{i<|d|}B_i^\boxed{i}f^\boxed{0}t\bigoplus_{i<|d|}b_i^\boxed{i}\sigma_f\Big)^m\\
&=\bigoplus_{i<|c|}A_i^\boxed{i}t^{-m}\bigoplus_{i<|c|}a_i^\boxed{i}\bigoplus_{i<|d|}B_i^\boxed{i}(f^\boxed{0}t)^m\bigoplus_{i<|d|}b_i^\boxed{i}\sigma
\end{split}
\end{equation*}
which simplifies to the expression in (\ref{alpham}).
\end{proof}

\begin{not*} For any $m, n \in \N$ and $g\in\G$, let $\beta(g, m, n):=\alpha_e^n(\alpha_e^{-m}\alpha_g^m)\alpha_e^{-n}$.
\end{not*}
\begin{lem} \label{formofbeta} Let $m, n \in \N$. Then $\beta(f, m, n)$ is of the form
\begin{equation*}
\bigoplus_{i<|c|}A_i^\boxed{i}\bigoplus_{i<|c|}a_{i}^\boxed{i+m-n}\bigoplus_{i<|d|}B_{i}^\boxed{i+m-n}\bigoplus_{i=1}^{m}f^\boxed{i-n}\bigoplus_{i<|d|}b_{i}^\boxed{i-n}\bigoplus_{i<|c|}A_{i}^\boxed{i-n}\bigoplus_{i<|c|}a_{i}^\boxed{i}\sigma'
\end{equation*}
where $\sigma' \in\bigoplus_{x \in X\setminus y\langle t\rangle}G_x$.
\end{lem}
\begin{proof}
Routine computations, together with $\sigma':=(t\sigma_e)^n\sigma(t^{-1}\sigma_e)^n$, yield the result.
\end{proof}
We now carefully choose $m, n \in \N$ so that $\beta(f, m, n)$ is an element of a $\Gamma_z$-set for some $z \in yH_X$. The importance of our particular choice is that the $z \in X$ will be independent of our choice of $f\in\G$, implying that there will be a $\Gamma_z$-set in $\langle \beta(g, m, n)\;:\;g \in \G$ and $m, n \in \N\rangle$.
\begin{lem}\label{buildinggammas} Given a $g\in\G$, there exist $m_g, n_g\in \N$ such that $\beta(g, m_g+n_g, n_g)$ is of the form $\bigoplus_{x\in X}g_x^{(x)}$ with $g_x^{(x)}\in G_x$ for each $x\in X$ and $g_{yt^c}^{(yt^c)}=g$. In particular, $\langle S_H'\cup (\G t)'\cup\{t\}'\rangle$ contains $\Gamma$, a $\Gamma_{yt^c}$-set.
\end{lem}
\begin{proof} We continue with our notation involving $f\in \G$. Lemma \ref{formofbeta} states, for any $m, n \in \N$, that $\beta(f, m,n)$ is of the form
\begin{equation*}
\bigoplus_{i<|c|}A_i^\boxed{i}\bigoplus_{i<|c|}a_{i}^\boxed{i+m-n}\bigoplus_{i<|d|}B_{i}^\boxed{i+m-n}\bigoplus_{i=1}^{m}f^\boxed{i-n}\bigoplus_{i<|d|}b_{i}^\boxed{i-n}\bigoplus_{i<|c|}A_{i}^\boxed{i-n}\bigoplus_{i<|c|}a_{i}^\boxed{i}\sigma'
\end{equation*}
where $\sigma' \in\bigoplus_{x \in X\setminus y\langle t\rangle}G_x$. Let $m_f\in \N$ be chosen such that $m_f>c+\max\{c, d\}$. This means that $m_f\ge c+1$, $-c+m_f>c$, and $-d+m_f>c$, all of which impact on certain summands appearing in $\beta(f, m, n)$.

Next, let $n_f\in \N$ be chosen such that $d-n_f<c$. The element $\beta(g, m_f+n_f, n_f)$ now has, for some $\sigma' \in\bigoplus_{x \in X\setminus y\langle t\rangle}G_x$, the form
\begin{equation*}
\bigoplus_{i<|c|}A_i^\boxed{i}\bigoplus_{i<|c|}a_{i}^\boxed{i+m_f}\bigoplus_{i<|d|}B_{i}^\boxed{i+m_f}\bigoplus_{i=1}^{m_f+n_f}f^\boxed{i-n_f}\bigoplus_{i<|d|}b_{i}^\boxed{i-n_f}\bigoplus_{i<|c|}A_{i}^\boxed{i-n_f}\bigoplus_{i<|c|}a_{i}^\boxed{i}\sigma'.
\end{equation*}
By considering each summand, together with the conditions placed on $m_f, n_f \in \N$, we see that $\beta(f, m_f+n_f, n_f)$ is an element $\bigoplus_{x \in X}g_x^{(x)}$ with $g_x^{(x)}\in G_x$ for each $x \in X$ and $g_{yt^c}^{(yt^c)}=f$. Since the constant $c$ relates only to $\alpha_e$, our argument applies to any $g\in\G$.
\end{proof}
\begin{proof}[Proof of Proposition \ref{nottorsionprop}] Since our choice of $y$ and $t$ were arbitrary, we can apply Lemma \ref{buildinggammas} to each $i\in I$ to obtain that
$$\langle S_H'\cup\big(\bigcup_{i \in I}\G_i'\cup(\G_it_i)'\cup\{t_i\}'\big)\rangle\supseteq\langle S_H'\cup\big(\bigcup_{i \in I} \G_i'\cup\Gamma_i\big)\rangle$$
which equals $G\wr_XH$ by Lemma \ref{thestrategy}.
\end{proof}

\bibliographystyle{amsalpha}
\def\cprime{$'$}
\providecommand{\bysame}{\leavevmode\hbox to3em{\hrulefill}\thinspace}
\providecommand{\MR}{\relax\ifhmode\unskip\space\fi MR }
\providecommand{\MRhref}[2]{%
 \href{http://www.ams.org/mathscinet-getitem?mr=#1}{#2}
}
\providecommand{\href}[2]{#2}
\end{document}